\documentclass[12pt, oneside, reqno]{amsart}

\usepackage{amsmath, amssymb}
\usepackage[frame,cmtip,arrow,matrix,line,graph,curve]{xy}
\usepackage{graphpap, color, paralist, pstricks}
\usepackage[mathscr]{eucal}
\usepackage[pdftex]{graphicx}
\usepackage[pdftex,colorlinks,backref=page,citecolor=blue]{hyperref}
\usepackage{tikz}
\usepackage{kotex}
\usepackage{array}
\usepackage{pgfplots}
\usepackage{mathrsfs}
\usetikzlibrary{arrows} 
\usepackage{tikz-cd}
\usepackage{mathtools}

\newcolumntype{P}[1]{>{\centering\arraybackslash}p{#1}}
\newcolumntype{M}[1]{>{\centering\arraybackslash}m{#1}}

\newtheorem{theorem}{Theorem}[section]
\newtheorem{proposition}[theorem]{Proposition}
\newtheorem{corollary}[theorem]{Corollary}
\newtheorem{lemma}[theorem]{Lemma}

\theoremstyle{definition}
\newtheorem{definition}[theorem]{Definition}
\newtheorem{example}[theorem]{Example}

\newtheorem{remark}[theorem]{Remark}

\newtheorem{notation}[theorem]{Notation}

\newcommand{\PP}{\mathbb{P}}

\newcommand{\CC}{\mathbb{C}}
\newcommand{\FF}{\mathbb{F}}
\newcommand{\KK}{\mathbb{K}}
\newcommand{\RR}{\mathbb{R}}

\newcommand{\HH}{\mathbb{H}}

\newcommand{\cO}{\mathcal{O} }
\newcommand{\cA}{\mathcal{A} }

\newcommand{\cC}{\mathcal{C} }
\newcommand{\cD}{\mathcal{D} }
\newcommand{\cE}{\mathcal{E} }
\newcommand{\cF}{\mathcal{F} }
\newcommand{\cG}{\mathcal{G} }
\newcommand{\cH}{\mathcal{H} }
\newcommand{\cI}{\mathcal{I} }

\newcommand{\cM}{\mathcal{M} }
\newcommand{\cN}{\mathcal{N} }

\newcommand{\cQ}{\mathcal{Q} }
\newcommand{\cS}{\mathcal{S} }
\newcommand{\cT}{\mathcal{T} }
\newcommand{\cU}{\mathcal{U} }
\newcommand{\cV}{\mathcal{V} }
\newcommand{\cW}{\mathcal{W} }
\newcommand{\cZ}{\mathcal{Z} }

\renewcommand{\bar}[1]{\overline{#1}}

\newcommand{\rA}{\mathrm{A} }

\newcommand{\rH}{\mathrm{H} }
\newcommand{\rM}{\mathrm{M} }

\newcommand\bD{\mathbf{D}}

\newcommand\bF{\mathbf{F}}

\newcommand\bH{\mathbf{H}}
\newcommand\bK{\mathbf{K}}
\newcommand\bM{\mathbf{M}}

\newcommand\bR{\mathbf{R}}
\newcommand\bQ{\mathbf{Q}}

\newcommand{\proj}{\mathrm{Proj}\;}
\newcommand{\im}{\mathrm{im}}

\def\Sym{\mathrm{Sym} }
\def\Hom{\mathrm{Hom} }
\def\Ext{\mathrm{Ext} }

\def\Gr{\mathrm{Gr} }

\def\SL{\mathrm{SL}}

\def\git{/\!/ }

\def\lr{\rightarrow}

\newcommand{\rank}{\mathrm{rank}\, }

\newcommand{\ses}[3]{0\lr{#1}\lr{#2}\lr{#3}\lr 0}

\newcommand\rk{\mathrm{rk}}

\hypersetup{%
pdftitle={A smooth compactification of rational curves},%
pdfauthor={Kiryong Chung, Jeong-Seop Kim},%
pdfkeywords={Rational quartic curves, Projective bundle, Equivariant Kodaira-Spencer map, Elementary modification},%
citecolor=blue,%
linkcolor=blue,%
}

\begin{document}

\title[A smooth compactification of rational curves]{An $\text{SL}(3,\mathbb{C})$-equivariant smooth compactification of rational quartic plane curves}
\author{Kiryong Chung}
\address{Department of Mathematics Education, Kyungpook National University, 80 Daehakro, Bukgu, Daegu 41566, Korea}
\email{krchung@knu.ac.kr}

\author{Jeong-Seop Kim}
\address{School of Mathematics, Korea Institute for Advanced Study, 85 Hoegiro, Dongdaemun-gu, Seoul 02455, Korea}
\email{jeongseop@kias.re.kr}


\begin{abstract}
Let $\mathbf{R}_d$ be the space of stable sheaves $F$ which satisfy the Hilbert polynomial $\chi(F(m))=dm+1$ and are supported on rational curves in the projective plane $\mathbb{P}^2$. Then $\mathbf{R}_1$ (resp. $\mathbf{R}_2$) is isomorphic to $\mathbf{R}_1\cong\mathbb{P}^2$ (resp. $\mathbf{R}_2\cong \mathbb{P}^5$). Also it is very well-known that $\mathbf{R}_3$ is isomorphic to a $\mathbb{P}^6$-bundle over $\mathbb{P}^2$. In special $\mathbf{R}_d$ is smooth for $d\leq 3$. But for $d\geq4$ case, one can imagine that the space $\mathbf{R}_d$ is no more smooth because of the complexity of boundary curves. In this paper, we obtain an $\mathrm{SL}(3,\mathbb{C})$-equivariant smooth resolution of $\mathbf{R}_4$ for $d=4$, which is a $\mathbb{P}^5$-bundle over the blow-up of a Kronecker modules space.
\end{abstract}

\maketitle

\section{Motivation and results}
We work on the field $\CC$ of complex numbers. 
\subsection{Rational curves and its moduli spaces}
Two ways are well-known to compactify rational curves space.
Let $X$ be a smooth projective variety with a fixed embedding $X\subset \PP^r$.
\begin{itemize}
\item (\cite{FP97}) Let $\cM_d(X)$ be the moduli space of \emph{stable} maps $[f:C\lr X]$ in $X$ with degree $\deg(f)=d$ and genus $g(C)=0$ .
\item (\cite{Sim94}) Let $\bM_{d}(X)$ be the moduli space of \emph{stable} sheaves $F$ in $X$ with Hilbert polynomial $\chi(F(m))=dm+1$.
\end{itemize}
The first (resp. second) one $\cM_d(X)$ (resp. $\bM_d(X)$) is often called by the Konstevich space (resp. Simpson space). The first moduli space $\cM_d(X)$ for a convex variety $X$ has at worst a finite group quotient singularity (i.e., orbifold singularity) (\cite[Theorem 2]{FP97}). Both spaces play a crucial role in counting of the number of curves, as they are known as Gromov-Witten (GW) and Donaldson-Thomas (DT) invariants (and also descendant invariants) respectively. Besides this aspect, it has played an important role in birational geometry (specially, log minimal model program or Fano/hyperk\"ahler geometry), and the first-named author studied the birational relationship between these spaces $\cM_d(X)$ and $\bM_d(X)$, $d\leq 3$ in a various setting (\cite{CK11, CHL18, Chu22}). To analyze the geometric structure of the Simpson space $\bM_d(X)$ in higher dimensional variety $X$ (and degree $d$), it is necessary to study the most basic case, namely, $X=\mathbb{P}^2$. For example, as shown in \cite{FT04, CCM16}, the moduli space $\bM_d(\PP^3)$, $d\leq 4$ consists of (several) irreducible components such that the intersection part is generically a fibration structure over the Grassimannian $\Gr(3,4)$ of planes in $\PP^3$.
In the viewpoint of enumerative geometry, the moduli space $\cM_d(\PP^2)$ provides a solution of a long standing problem called by \emph{Konstevich's formula} (\cite{FP97}). In detail, let $N_d$ be the number of
degree $d$ rational curves passing through $3d-1$ given points in general position in $\PP^2$. In 1995, through the work of Kontsevich (\cite{Kon95}), it has been found a recursive formula to compute $N_d$:
\[
N_d=\sum_{d_1+d_2=d}N_{d_1}N_{d_2}d_1^2d_2(d_2{3d-4 \choose 3d_1-2} -d_1{3d-4 \choose 3d_1-1}).
\]
Clearly $N_1=N_2=1$. Even the case of lower degrees has a very long history. In the mid-1800s, Steiner (resp. Zeuthen) showed that $N_3 = 12$ (resp. $N_4 = 620$) (\cite{Ste54, Zeu73}). Also $N_5$ was computed in middle of 1980.

On the other hand, prior to Kontsevich's work, research on the geometry of plane curves was focused on the study of Severi varieties $V_{d,\delta}$ (so called, the space of curves with prescribed degree $d$ and the number $\delta$ of nodes), and J. Harris made significant contributions and provided direction in this area. One of his notable achievements was proving the irreducibility of the Severi variety $V_{d,\delta}$ (\cite{Har86}). Related with this, in \cite{DH88b}, the authors mentioned the need for a smooth space with a moduli meaningful points. As one may expect, the boundary points of plane curves space may have bad singularity as possible and thus it is very hard to construct a smooth moduli space globally.

Although the management of orbifold singularities of the Kontsevich space has led to the development of a rich theory (so called, \emph{perfect obstruction} theory \cite{BF97}) in enumerative geometry, the fundamental question about smooth compactification of $V_{d,\delta}$ has been a central concern by following the monumental work of Hironaka (\cite{Hir64}). In this stream, we propose a smooth compactification of the space of degree $4$ rational plane curves, which is the first non-trivial case in $\PP^2$.

\subsection{Main results}
As a continued work of \cite{Chu21}, we recall the result and unsolved part in the view-point of \emph{resolution of singularities}.
Since a general plane rational quartic curve $C$ has three nodal points $Z$ by degree-genus formula, we consider space of pairs $(C,Z)$ in the relative Hilbert scheme of three points over the complete linear system of degree $4$. Let us denote $\bR_4^{h}$ by the closure of the locus of the rational quartic curves in this relative Hilbert scheme. Then there is a canonical projection map $\bR_4^{h}\lr \bH[3]$ into the Hilbert scheme of three points on $\PP^2$. On the hand, as identifying $\bR_4^{h}$ with the space of a stable pairs $(s, \cE xt^1(I_{Z|C}(1),\omega_{\PP^2}):=F)$, $s\in\rH^0(F)$ and performing the wall-crossing in the sense of variation of geometric invariant quotients (\cite{Tha96, He98}), we obtain a smooth contraction of the locus of rational quartic curves with three collinear points. In fact, let $\bR_4$ be the closure of the locus of rational quartic curves in the Simpson space $\bM_4(\PP^2)$, then we have a commutative diagram as follows.
\[
\xymatrix{\bR_4^{h}\ar[rr]^{\text{sm. blow-down}}\ar[d]&&\bR_4\ar[d]^{p_{s}}\\
\bH[3]\ar[rr]^{\pi}
&&\bK}
\]
Here $\bK$ is a \emph{Kronecker modules space}, which is a birational contraction of $\bH[3]$ along the locus $\widetilde{\bD}_3$ of collinear three points (For an explicit definition, see Section \ref{hilbpt}). Let $\bD_3$ be the image of $\widetilde{\bD}_3$ along the map $\pi:\bH[3]\lr\bK$. The main result of the paper \cite{Chu21} is to describe the fiber of the map $p_{s}:\bR_4\lr \bK$ over $\bD_3$. That is, each fiber of $p_s$ over $\bD_3$ parameterizes the stable sheaf which is completely determined by its support (For detail, see (2) of Proposition \ref{sum2}). It is remarkable that our proof is to use the elementary modification of the direct image sheaf of the universal family of $\cM_4(\PP^2)$. From this, we can conclude that $p_s$ is a $\PP^5$-fiberation over $\bK$ outside of the subspace $\bD_5$. Here, $\bD_5\subset \bK\setminus \bD_3=\bH[3]\setminus \widetilde{\bD}_3$ parametrizes what are called \emph{non-curvilinear} points in $\PP^2$ (Definition \ref{degdef}). Over $\bD_5$, the fiber of $p_s$ parameterizes stable sheaves whose support are rational quartic curves with a $D_4$-singularity (Figure \ref{fig1}) and thus $p_s|_{p_s^{-1}(\bD_5)}$ is a $\PP^8$-fiberation over $\bD_5$ ((1) of Proposition \ref{sum2}). The aim of this paper is to resolve this singular locus by the flattening of subschemes.

\begin{theorem}[\protect{Theorem \ref{mainthm}}]\label{mainthm0}  
Let $\bR_4$ be the closure of the space of rational quartic curves in the moduli space $\bM_4(\PP^2)$ of stable sheaves in $\PP^2$ with Hilbert polynomial $4m+1$. There exists an $\mathrm{SL}(3,\mathbb{C})$-equivariant smooth resolution $\widetilde{\bR}_4$ of $\bR_4$.
\end{theorem}
In fact, $\widetilde{\bR}_4$ is a projective bundle over the blown-up space $\widetilde{\bK}$ of $\bK$ along $\bD_5$.
 As a corollary, we set up the recipe for the computation of Chow ring of the space $\widetilde{\bR}_4$ (Proposition \ref{chowset}) and confirm the flip picture (Corollary \ref{marcor}) suggested by Maruyama (\cite{Mar73}).

The main idea of the proof of Theorem \ref{mainthm0} is as follows. Motivated from the canonical form of the planar rational quartic curve, we consider a rank $3$ vector bundle $\cF$ on $\bK$ whose fiber is generated by degree $2$ parts of the defining equation of three points on $\PP^2$. By taking the projectivization of the second symmetric power $\Sym^2\cF=:\cV$, we have a birational map \[\Phi:\PP(
\cV)\dashrightarrow \PP(\cU)\] (induced by the inclusion map $\cV\subset \cU$) over $\bK$ where the generic fiber of the bundle $\cU$ is generated by the quartic forms passing through three points in $\PP^2$. We show that  the non-torsion free locus of the quotient sheaf $\cU/\cV$ is exactly the locus $\bD_5(\cong\PP^2)$ of the non-curvilinear points. Now by considering the family of square ideals of three points in $\PP^2$, we see that $\bD_5$ parameterizes the locus of non-flat family of subschemes in $\PP^2$. After the blow-up of $\bK$ along $\bD_5$ and flattening of subschemes, we extend the map $\Phi$ as a birational morphism. All of the aforementioned statements are in the $\mathrm{SL}(3,\mathbb{C})=:\SL_3$-equivariant setting. Additionally, it is remarkable that the Kodaira-Spencer map in the proof Proposition \ref{mainprop2} is an equivariant one and thus we easily detect the flat limits parametrized by the blown-up space $\widetilde{\bK}$.
\subsection{Organization of the paper} In Section \ref{setting}, we collect the well-known facts about the Hilbert scheme of points, Bridgeland's stability, and the symmetric power of coherent sheaves. In Section \ref{babycase}, as a warm-up, we study the compactification of the space of rational cubic curves in $\PP^2$. Lastly, in Section \ref{d4case}, we provide a proof of Theorem \ref{mainthm0}.

\subsection*{Notations and conventions}
\begin{itemize}
\item Let $W_{n+1}$ be an $(n+1)$-dimensional complex vector space. We shortly write $\PP^{n}= \PP W_{n+1}$.
\item  The \emph{projective bundle} of a locally free sheaf $\cF$ over $X$ is defined by 
\[
\PP(\cF):=\proj(\mathrm{Sym}^{\bullet} (\cF^*))\lr X, 
\]
Hence $\PP(\cF)$ is the space of one-dimensional subspaces of $\cF$. 
\item $\bM_4 :=\bM_{d}(\PP^2)$: The moduli space of stable sheaves $F$ with Hilbert polynomial $\chi(F(m))=4m+1$.
\item $\bK := \bK(3;2,3)$: The moduli space of quiver representations of $3$-Kronecker quiver with the dimension vector $(2, 3)$ (Notation \ref{krodef}).
\item $\bQ_4:=\PP(\cU)$, where $\cU$ is a rank $12$ vector bundle over $\bK$ (Definition \ref{defqb}).
\item $\bH[n]:=\mathrm{Hilb}^n(\PP^2)$: The Hilbert scheme of zero-dimensional subschemes of length $n$ on $\PP^2$.
\item $\Rightarrow$ (resp. $\xhookrightarrow{}{}$, $\twoheadrightarrow$) denotes a bundle (resp. injective, surjective) morphism of varieties or locally free sheaves.
\item We often omit the pullback symbol in the pullback of locally free sheaves (or vector bundles) along a morphism.
\end{itemize}
\subsection*{Acknowledgements}
The topic in this paper was originally proposed by D. Chen when K. Chung was graduating (\cite[Chapter 6]{Chu11}), which provided an opportunity to explore various compactifications of the space of rational curves. K. Chung expresses deep gratitude to D. Chen. After writing the paper, the authors noticed some overlap between Section \ref{modsec} and \cite[Section 6]{LFV21}. We apologize for any redundancy with \cite{LFV21}, although our results are based on the same perspective as those in \cite{Chu21}.
\section{Preliminary}\label{setting}
In this section, we collect very well-known facts about the Hilbert scheme of points, the flip of moduli of stable objects and symmetric power of coherent sheaves.
\subsection{The Kronecker and Hilbert scheme}\label{hilbpt}
The moduli space of Kronecker modules is given by a  geometric invariant theoretic (GIT) quotient as follows. Let $\bK(3;2,3)$ be the moduli space of quiver representations of $3$-Kronecker quiver
\[
\xymatrix{\bullet\ar@/^/[r] \ar@/_/[r]\ar[r]&\bullet\\}
\]
with dimension vector $(2, 3)$. It can be identified with the space of isomorphism classes of stable sheaf homomorphisms
\[
\cO_{\PP^2}(-2)^{\oplus 2}\longrightarrow \cO_{\PP^2}(-1)^{\oplus 3}
\]
up to the action of the automorphism group $G := \mathrm{GL}_{d-2}\times \mathrm{GL}_{d-1}/\CC^*$. Thus for two vector spaces $E$ and $F$ of dimension $2$ and $3$ respectively and $W_3^* = \rH^0(\cO_{\PP^2}(1))$, the moduli space of Kronecker modules is given by the GIT quotient (\cite{Kin94})
\[\bK(3;2,3) := \Hom(F,W_3^{*}\otimes E)\git G.\]
\begin{notation}\label{krodef}
Let $\bK := \bK(3;2,3)$.
\end{notation}
 The space $\bK$ carries two bundles $\cE$ and $\cF$ with $\rank \cE=2$ and $\rank \cF=3$ such that $\wedge^2\cE\cong \wedge^3\cF$. Furthermore, the bundle $\cF$ provides an embedding map $i:\bK\subset \Gr(3, \Sym^2W_3^*)$
via the inclusion map 
\begin{equation}\label{univinc}
\cF\subset \Sym^2W_3^*\otimes \cO_{\bK}.
\end{equation}
On the other hand, there exists a \emph{universal complex}
\begin{equation}\label{univseq}
\KK^{\bullet}:\cE \boxtimes \cO_{\PP^2}(-1)\stackrel{\phi}{\hookrightarrow} \cF\boxtimes \cO_{\PP^2}\stackrel{\psi}{\longrightarrow} \cO_{\bK}\boxtimes \cO_{\PP^2}(2)
\end{equation}
over $\bK\times \PP^2$ such that $\phi$ is injective. The first map $\phi$ in the complex $\KK^{\bullet}$ in \eqref{univseq} provides the universal resolution of the ideal of the Hilbert scheme $\bH[3]$ of three points on $\PP^2$ excepting the locus $\widetilde{\bD_3}$ of collinear three points.
Note that $\widetilde{\bD_3}$ is isomorphic to a $\PP^3$-bundle over the space $\Gr(1,W_3^*)$ of lines, where each fiber $\PP^3$ parameterizes triple points on a line. In fact, the image of the map
\[
\pi:\bH[3]\lr \Gr(3,\Sym^2W_3^*),\;\; [Z]\mapsto \rH^0(I_Z(2))
\]
is $\bK$ and the map $\pi$ is a smooth blow-down with the exceptional divisor $\widetilde{\bD}_3$ (\cite{LQZ03}). Also the restriction map of $\pi$ along $\widetilde{\bD}_3$ contracts the fiber $\PP^3$.
\begin{notation}
Let us denote the image of $\widetilde{\bD}_3$ by $\bD_3 (\cong \Gr(1,W_3^*)$) in $\bK$.
\end{notation}
For each closed point $[Z]\in \bK\setminus \bD_3$, the fibers of locally free sheaves $\cE$ and $\cF$ at $[Z]$ are given by
\[\begin{split}
\cF_{[Z]}&=\rH^0(\PP^2, I_Z(2))\subset \rH^0(\PP^2, \cO_{\PP^2}(2))=\Sym^2W_3^*,\\
\cE_{[Z]}&=\text{ker}\left\{\cF_{[Z]}\otimes\rH^0(\PP^2, \cO_{\PP^2}(1))\stackrel{m}{\longrightarrow}\rH^0(\PP^2, \cO_{\PP^2}(3))\right\},
\end{split}
\]
where $m$ is the restriction of the canonical multiplication map $m: \Sym^2W_3^*\otimes W_3^*\lr \Sym^3W_3^*$. But the complex $\KK^{\bullet}$ in \eqref{univseq} is not acyclic (i.e., $h^0\neq0$) over $\bD_3$. For each closed point $[l]\in \bD_3$, let $\text{coker}(\phi)_{[l]}:=E_{[l]}$ be the cokernel of $\phi_{[l]}$. Then $E_{[l]}$ fits into the short exact sequence:
\begin{equation}\label{eq011}
\ses{\cO_{l}(-1)}{E_{[l]}}{\cO_{\PP^2}(1)}.
\end{equation}
Furthermore, the map $\psi_{[l]}$ in the complex $\KK^{\bullet}$ decomposes into
\begin{center}
\begin{equation}\label{fact1}
\begin{tikzcd}
{\cF_{[l]}\otimes \cO_{\PP^2}} \arrow[rr, "{\psi_{[l]}}"] \arrow[rd, two heads] &                                           & \cO_{\PP^2}(2)                 \\
                                               & {E_{[l]}} \arrow[ru] \arrow[r, two heads] & \cO_{\PP^2}(1). \arrow[u, hook]
\end{tikzcd}
\end{equation}
\end{center}
\begin{remark}
The complex in \eqref{univseq} is a special case of the following one.
Let $\phi: \cE_0\lr \cE_1$ be a morphism of locally free sheaves $\cE_0$ and $\cE_1$ on a smooth projective variety $X$ with $\rank\cE_0+1=\rank \cE_1=n$. There exists an \emph{Eagon-Northcott complex} associated to the dual map $\phi^{t}: \cE_1^*\lr \cE_0^*$ (\cite{EN62}). By using  an isomorphism $\wedge^{n-1}\cE_1\cong \cE_1^*\otimes\wedge^{n}\cE_1$, one can have a complex
\[
\cE_0\stackrel{\phi}\longrightarrow \cE_1\stackrel{\wedge^{n-1}\phi^{t}}{\longrightarrow}(\wedge^{n}\cE_1)\otimes\wedge^{n-1}\cE_0^*.
\]
\end{remark}

\subsection{Moduli space of stable objects and its flip}\label{fliplocus}
The moduli space of semistable objects on $\PP^2$ has been extensively studied in the context of birational geometry (\cite{BMW14}). In this paper, we briefly discuss the chamber structure of Bridgeland's stability space and its wall-crossings.

We define the potential function for $s \in \RR$ and $t > 0$ as follows:
\[
	Z_{s,t}(E)=- \int_{\PP^2} e^{-(s+it)H} \cdot ch(E)
	= -\left(ch_{2}- sc_{1} + \frac{1}{2}r(s^{2}-t^{2})\right)
	+ it(c_{1}-rs),
\]
where $(r, c_{1}, ch_{2}) = (r(E), c_{1}(E),ch_{2}(E))$ and $H=c_1(\cO_{\PP^2}(1))$. We can determine the stability of objects in the heart $\cA_{s}$ of a certain $t$-structure (depending only on $s$) in $\mathrm{D}^b(\PP^2)$ using the slope function:
\[
\mu_{s,t} : \mathrm{K}(\PP^{2}) \to \RR,\;	\mu_{s,t}(E)=
	-\frac{\mathrm{Re}(Z_{s,t}(E))}{\mathrm{Im}(Z_{s,t}(E))}
	= \frac{ch_{2} - sc_{1}+\frac{1}{2}r(s^{2}-t^{2})}{t(c_{1}-rs)}.
\]
Also the stability space (identified with the upper half plane $\HH = \{(s,t)\in \RR^{2}\;|\; t> 0\}$) can be decomposed into a finite number of chambers. Let $\rM_{\mu_{s,t}}(v)$ be the moduli space of Bridgeland stable objects in $\cA_{s}$ with a fixed type $v = (r, c_{1}, ch_{2}) \in \rA^{*}(\PP^{2})$. The space $\rM_{\mu_{s,t}}(v)$ varies only if $(s,t)$ moves from one chamber to another.

In our case, let us fix $v(F)=(0, 4, -5)\in\rA^{*}(\PP^{2})$ (and thus $\chi(F(m)) = 4m+1$). 
\begin{definition}\label{defqb}
Let $\cU:=p_* (\text{coker}(\phi)\boxtimes \cO_{\PP^2}(2))$ in \eqref{univseq} be the locally free sheaf on $\bK$, where $p:\bQ_4:=\PP(\cU)\lr \bK$ is the canonical bundle morphism.
\end{definition}
\begin{remark}
The projectivized bundle $\PP(\cU)$, when restricted to the open subset $\bK\setminus\bD_3$, is isomorphic to the relative Hilbert scheme parameterizing three non-collinear points on plane quartic curves, as shown in \cite[Section 4.4]{He98}.
\end{remark}
\begin{proposition}[\protect{\cite[Lemma 6.9]{Woo13}}]
Under the above notations and definitions,
\begin{enumerate}
\item The walls dividing chambers in $\HH$ are semi-circles with the common center $(-\frac{5}{4},0)$.
\item There is only one numerical wall when the radius of the semicircle is $R = \frac{7}{4}$.
\item There are two different moduli spaces of stable objects:
\[\rM_{\mu_{s,t}}(v)\cong
\begin{cases}
\bQ_4,\;t \ll 1,\\
\bM_4,\; t\gg 1.
\end{cases}
\]
\end{enumerate}
\end{proposition}
The universal subbundle homomorphism $\cO_{\bQ_4}(-1)\hookrightarrow p^*\cU$ determines a universal family of stable complexes
\begin{equation}\label{univ1}
\FF_4:=[\cO_{\bQ_4\times \PP^2}\rightarrow \pi_1^*\cO_{\bQ_4}(1)\otimes (p\times\text{id})^*(\text{coker}(\phi)\boxtimes \cO_{\PP^2}(2))].
\end{equation}
As modifying of the complex $\FF$ in \eqref{univ1}, there exists a flip diagram between $\bQ_4$ and $\bM_4$ such that the flipping loci are given by
\begin{center}
\begin{tikzcd}
\bQ_4 \arrow[rr, dashed]               &      ^{\text{flip}}           & \bM_4 \arrow[ll, dashed]                    & {} \\
F^{-}=\text{Fl}(W_3) \arrow[rd,Rightarrow]  \arrow[u, hook] &                & F^{+}=\cC_4 \arrow[ld,Rightarrow] \arrow[ru, phantom] \arrow[u, hook] &    \\
                                            & \PP^2&                                                      &   
\end{tikzcd}
\end{center}
\begin{itemize}
\item $\bF^{-}$ is isomorphic to the \emph{universal lines} $\text{Fl}(W_3)=\{(q,l)\mid q\in l\in\Gr(1,W_3^*)\}$ and
\item $\bF^{+}$ is isomorphic to the \emph{universal quartic curves} $\cC_4=\{(q,C)\mid q\in C\in |\cO_{\PP^2}(4)|\}$.
\end{itemize}
We remark that the flipping locus $\bF^{+}$ parameterizes stable sheaves $F$ satisfying $h^0(F)=2$.
The complex $F:=[\cO_{\PP^2}\stackrel{s}{\lr} I_{Z|\PP^2}(4)]\in \bQ_4\setminus \bF^{-}$ bijectively associates its derived dual $F^{\vee}\in \bM_4\setminus \bF^{+}$.
\begin{proposition}\label{conh1}
Let $[l]\in \bD_3=\Gr(1,W_3^*)$ and
\[s=(s_1,s_2)\in \rH^0(E_{[l]}(2))=\rH^0(\cO_l(1))\oplus \rH^0(\cO_{\PP^2}(3))\]
such that $E_{[l]}$ in \eqref{eq011}. Then,
\begin{enumerate}
\item $\bF^{-}$ is contained in $p^{-1} (\bD_3)$.
\item $s_2=0$ if and only if $[\FF_4|_{([l],s)\times\PP^2}]\in \bF^{-}$, where the stable complex fits into the exact triangle
\[
0\lr [\cO_{\PP^2}\stackrel{s_1}{\longrightarrow} \cO_l(1)]\lr\FF_4|_{([l],s)\times\PP^2}\lr  \cO_{\PP^2}(3)[1]\lr0.
\]
\item If $s_2\neq 0$, then the derived dual $\FF_4|_{([l],s)\times\PP^2}^{\vee}$ is quasi-isomorphic to a stable sheave $F$ fitting into the exact sequence
\begin{equation}\label{eq01}
\ses{\cO_C}{F}{\cO_{l}}
\end{equation}
for a cubic curve $C=\text{Supp}(\text{coker}(s_2))$. Thus $h^0(F)=1$.
\end{enumerate}
\end{proposition}
\begin{proof}
(1) and (2): \cite[Lemma 3.9 and Lemma 3.7]{CM17}.

(3) When the diagram (9) in \cite{CM17} is read in reverse, one can obtain the exact sequence \eqref{eq01}.
\end{proof}
Recall that $\bR_4$ is the closure of the space of rational quartic curves in $\bM_4$. From \cite[Proposition 3.3]{Chu21}, $h^0(F)=1$ for each point $[F]\in \bR_4$ and thus there is an embedding $\bR_4\subset\bM_4\setminus \bF^{-}$. 
\begin{proposition}[\protect{\cite[Proposition 3.6 and Proposition 4.1]{Chu21}}]\label{sum2}
Let $\pi$ be the composition of the maps
\[\pi: \bR_4\subset \bM_4\setminus\bF^{-}\cong \bQ_4\setminus \bF^{+}\lr \bK.\]
Then the map $\pi$ is a generic $\PP^5$-fibration over $\bK$ and the special fiber is described as follows.
\begin{enumerate}
\item $\pi$ is a $\PP^8$-fibration map over the locus $\bD_5$ of non-curvilinear points (For the explicit definition, see Definition \ref{degdef}) such that the fiber $\PP^8$ parameterizes the plane quartic curves singular at a point parameterized by $\bD_5$.
\item For any $[l]\in\bD_3$, the fiber $\pi^{-1}([l])\cong\PP^5$ parametrizes a stable sheaf $F$ that fits into a \emph{unique} exact sequence of the form $\ses{\cO_{C\cdot l}}{F}{\cO_l}$, where $C$ is a conic parameterized by the fiber $\PP^5$ (cf. Proposition \ref{conh1}).
\end{enumerate}
\end{proposition}
The rational curves in (1) of Proposition \ref{sum2} generically have $D_4$-singularities (Figure~\ref{fig1}). For discussions on singularities of quartic plane curves, we may refer to \cite{Hui79} and \cite[Chapter 9]{CKPU13}. The main goal of this paper is to resolve this locus by means of elementary modification of sheaves (bundles).
\begin{figure}[h]
\begin{tikzpicture}[scale=0.9]
\begin{axis}[axis lines=none]
\addplot[domain=-70:70,samples=300,color=black,very thick,data cs=polar] (x,{sin(3*x)});
\end{axis}
\end{tikzpicture}
\vspace{-1.5em}
\caption{A rational quartic plane curve with a $D_4$-singularity}\label{fig1}
\end{figure}
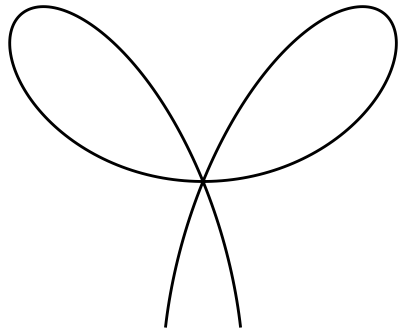
\subsection{Symmetric powers and its basics of $\SL_3$-representation}\label{stdsl3}
Let $\mathcal{F}^{\otimes k}$ be the tensor power of a coherent sheaf $\mathcal{F}$ on a smooth projective variety $X$. The $k$-th symmetric power of $\mathcal{F}$, denoted by $\operatorname{Sym}^k\mathcal{F}$, is defined as the quotient of the tensor power $\mathcal{F}^{\otimes k}$ by the ideal sheaf generated by all elements of the form $s_1\otimes \cdots \otimes s_k - \sigma(s_1)\otimes\cdots\otimes\sigma(s_k)$, where $s_1,\dots,s_k$ are local sections of $\mathcal{F}$ and $\sigma$ is any element of the symmetric group $S_n$ of $n$-letters. Let us denote the equivalence class by $s_1\cdot s_2\cdots \cdot s_k:=[s_1\otimes s_2\otimes \cdots \otimes s_k]\in \operatorname{Sym}^k\mathcal{F}$.
\begin{lemma}
Let $\cE$, $\cF$, and $\cH$ be coherent sheaves on $X$. Consider an exact sequence ${\cE}\stackrel{\delta}{\longrightarrow}{\cF}\stackrel{\epsilon}{\longrightarrow}{\cH}\longrightarrow 0$. Then, for all $k\geq 1$, we obtain an exact sequence 
\[
\cE\otimes \Sym^{k-1} \cF\stackrel{s(\delta)}{\longrightarrow}  \Sym^{k} \cF \stackrel{\text{sym}^k(\epsilon)}{\longrightarrow}\Sym^{k}\cH\lr0
\]
where $s(\delta)$ is defined by
\[
e\otimes(f_1\cdot f_2\cdots f_{k-1})\mapsto \sum_{i=1}^{k} (f_1\cdot f_2\cdots \delta(e) \cdots f_{k-1}).
\]
\end{lemma}
For our purpose, we need a standard $\SL_3$-representation theory related with symmetric power spaces. We summarize the contents of Section 13.3 in \cite{FH91}. Let $V$ be a vector space with dimension $\dim V=3$. Then the standard representation of $\Sym^2(\Sym^2 V)$ is isomorphic to the space
\[
\begin{split}
\Sym^2(\Sym^2 V)&\cong \Sym^4 V\oplus \Sym^2V^*\\
&\cong\Sym^4 V\oplus \Sym^2(\wedge^2 V)
\end{split}
\]
where the kernel of the multiplication map $m:\Sym^2(\Sym^2 V)\twoheadrightarrow \Sym^4 V$ is identified with $\text{ker}(m)=\Sym^2(\wedge^2V)\cong \Sym^2V^*$ under a fixed isomorphism $\wedge^2V\cong V^*\otimes \wedge^3V$. Here the embedding is given by \[\Sym^2(\wedge^2V)\subset\Sym^2(\Sym^2 V),\;(u\wedge v) \cdot (w\wedge z)\mapsto (u\cdot w)\cdot(v\cdot z)-(u\cdot z)\cdot (v\cdot w).\]
\begin{example}
Let $(I_Z)_2=\langle x^2, xy, y^2\rangle$ be the degree $2$ part of the ideal of the non-curvilinear point $Z$ supported at $\text{Supp}(Z)=\{[0:0:1]\}$. As a $\CC$-vector space, $(I_Z)_2$ is a subvector space of $\Sym^2V_3$ for $V_3=\text{span}_{\CC}\{x, y, z\}$. One may choose a basis of $\Sym^2(\langle x^2, xy, y^2\rangle)$ as 
\[\left\{x^2\cdot x^2,\; x^2\cdot(xy),\; x^2\cdot y^2,\; (xy)\cdot (xy),\; (xy)\cdot y^2,\; y^2\cdot y^2\right\},\]
which is a subvector space of $\Sym^2(\Sym^2 V_3)$. Note that $x^2\cdot y^2\neq (xy)\cdot (xy)\in \Sym^2(\Sym^2 V_3)$. However, under the canonical multiplication map $m:\Sym^2(\Sym^2 V_3)\twoheadrightarrow \Sym^4 V_3$, they have the same image, i.e., $m(x^2\cdot y^2)=m((xy)\cdot (xy))$. Therefore, we have 
\[(x\wedge y)\cdot(x\wedge y)=x^2\cdot y^2-(xy)\cdot (xy)\in \ker(m).\]
\end{example}
\section{Warm-up: Degree $3$ case}\label{babycase}
In this section, we describe the space $\bR_3$ as a projective bundle based on the results of \cite{CC11}. Even though the result is very elementary, the purpose of this writing is to provide motivation for the degree $4$ case, which will be discussed in Section \ref{d4case}.
\subsection{Rational cubic curves space and its universal complex} Consider a cubic curve $C$ in $\PP^2=\PP(W_3)$ defined by the cubic equation $f(x,y,z)=0$. 
A cubic curve $C$ is rational if it has a singular point by the degree-genus formula.
The cubic curve $C$ passes through the point $q=[0:0:1]$ if and only if $f(x,y,z)$ can be written in the form 
\begin{equation}\label{cubicform}
f(x,y,z)=g_3(x,y)+zg_2(x,y) +z^2g_1(x,y),
\end{equation}
where $g_i(x,y)$ is a homogeneous polynomial of $x, y$ and $\deg g(x,y)=i$. Furthermore, $C$ is singular at the point $q$ if and only if $g_1(x,y)=0$ in \eqref{cubicform}. We describe how one can construct the parameter space of rational cubic curves through the canonical form \eqref{cubicform}. Let $\cE_i$ be the locally free sheaf whose fiber corresponds to cubic forms vanishing of order $i$ at a point in $\PP^2$.
Let $\bD=\Gr(2,W_3^*)$ be the Grassmannian of points in $\PP^2$. By taking symmetric powers of the universal bundle sequence $\ses{\cD}{W_3^*\otimes \cO_{\bD}}{\cQ}$, we obtain a commutative diagram:
\begin{equation}\label{orddia}
\xymatrix{&&\\
&W_3^*\otimes \cQ^{\otimes2}\ar@{=}[r]&W_3^*\otimes \cQ^{\otimes2}\\
\Sym^3\cD\ar@{^{(}->}[r]\ar@{=}[d]&\Sym^3W_3^*\otimes \cO_{\bD}\ar@{->>}[r]\ar@{->>}[u]&\Sym^2W_3^*\otimes \cQ\ar@{->>}[u]\\
\Sym^3\cD\ar@{^{(}->}[r]&\cE_2\ar@{^{(}->}[u]\ar@{->>}[r]&\Sym^2\cD\otimes \cQ\ar@{^{(}->}[u],}
\end{equation}
such that the sheaf $\cE_2$ fits into the bottom row of the diagram \eqref{orddia}.
\begin{proposition}
Under the above notations and definitions,
\begin{enumerate}
\item
There exist inclusions of locally free sheaves on $\bD$:
\[\cE_2\subset \cE_1\subset \Sym^3W_3^*\otimes \cO_{\bD}.\]
Furthermore, the moduli space $\bM_3$ (resp. $\bR_3$)
 is isomorphic to
 \[\bM_3\cong \PP(\cE_1),\; (\text{resp}. \;\bR_3\cong\PP(\cE_2))\] and the quotient $\cE_1/\cE_2$ is isomorphic to a locally free sheaf $\cE_1/\cE_2\cong \cD\otimes \cO_{\bD}(2)$.
 \item The universal family of stable complex over $\bR_3$ is given by the derived dual
 \begin{equation}\label{univ2}
\FF_3:=[\cO_{\bR_3\times \PP^2}\rightarrow \pi_1^*\cO_{\bR_3}(1)\otimes (p\times\text{id})^*(\cI_{\cW}\boxtimes \cO_{\PP^2}(3))]^{\vee}
\end{equation}
where $\cW$ is the universal subscheme of a point over $\bD$ and $p:\PP(\cE_2)\lr \bD$ is the canonical projection map.
\end{enumerate}
 \end{proposition}
 \begin{proof}
 (1) One can easily see that the kernel of the composition map of the middle-upper map of the diagram \eqref{orddia} and the canonical surjective map $W_3^*\otimes \cO_{\bD}\twoheadrightarrow \cQ$ twisted by $\cQ^{\otimes2}$ is the locally free sheaf 
\[
\cE_1=\text{Ker}\left\{\Sym^3W_3^*\otimes \cO_{\bD}\twoheadrightarrow W_3^*\otimes \cQ^{\otimes2}\twoheadrightarrow \cQ^{\otimes3}\right\}.
\]
Thus we prove the inclusion relation. Also the isomorphisms come from the first paragraph of \cite{CC11}. The quotient sheaf $\cE_1/\cE_2$ is finally obtained by analyzing the canonical form in \eqref{cubicform}.

(3) One can construct the universal complex over $\bM_3$ by relativizing the following nested exact sequence
\[
\ses{\cO_{\PP^2}(-3)=I_{C|\PP^2}}{I_{q|\PP^2}}{I_{q|C}}
\]
of ideals for the pair $q\in C$ for some point $q$ and cubic curve $C$ in $\PP^2$. The claimed universal complex in \eqref{univ2} is nothing but the pull-back of this complex along the inclusion map $\PP(\cE_2)\times\PP^2\subset \PP(\cE_1)\times\PP^2$.
\end{proof}
\section{Degree $4$ case}\label{d4case}
In this section, we first construct a birational model of $\bR_4$, which is a projectivization of a symmetrized bundle on the Kronecker modules space (Proposition \ref{mainprop1}). Secondly, by studying the undefined loci and flattening of a family of subschemes (Proposition \ref{mainprop2}), we modify the bundle and thus obtain an extended birational morphism into $\bR_4$ (Theorem \ref{mainthm}). Finally, we explain the contraction of the exceptional locus into $\bR_4$. All of these works are an $\SL_3$-invariant one.
\subsection{Universal bundles, symmetrization and its projectivization}
In this subsection, we construct a birational map between a $\PP^5$-projective bundle over $\bK$ and the rational curves space $\bR_4$ (Proposition \ref{mainprop1}). The canonical form (\cite[Section 3.2]{Chu21}) of the rational quartic plane curve suggests that a natural candidate for such a bundle is the symmetric product $\Sym^2\cF$ of the universal subbundle $\cF$ in \eqref{univinc}. Let
\begin{equation}\label{symten}
\Sym^2\cF\subset \cF\otimes \cF
\end{equation}
be the inclusion map induced by $v\cdot w\mapsto \frac{1}{2}(v\otimes w+w\otimes v)$ for $v, w\in \cF$. As composing the inclusion map \eqref{symten} and the second map in $\psi$ in \eqref{univseq} tensored with $\cF\boxtimes(-)$, we have a morphism over $\bK\times \PP^2$:
\begin{equation}\label{comk}
k:\Sym^2\cF\subset \cF\otimes \cF \stackrel{\text{id}\otimes \psi}{\longrightarrow} \cF\boxtimes \cO_{\PP^2}(2).
\end{equation}
Now by the tensoring the line bundle $\cO_{\PP^2}(2)$ in the universal complex \eqref{univseq} and combining with the map $k$ in equation \eqref{comk}, we get a commutative diagram:
\begin{equation}\label{univcom}
\xymatrix{
\Sym^2\cF\ar[d]^{k}\ar@/_0.7pc/[drr]^{(\psi\otimes \text{id})\circ k}&&\\
\cF\boxtimes \cO_{\PP^2}(2)\ar[rr]_{\psi\otimes \text{id}}\ar@{->>}[dr]^{s}&&\cO_{\bK}\boxtimes \cO_{\PP^2}(4)\\
&\cM:=\text{coker}(\phi)\otimes \cO_{\PP^2}(2)\ar[ur]_{\bar{\psi}\otimes \text{id}}&
}
\end{equation}
over $\bK\times \PP^2$. The push-forward along the projection map $p:\bK\times \PP^2\lr \bK$ of the diagram \eqref{univcom} provides
\begin{equation}\label{pushdown}
\xymatrix{p_*(\Sym^2p^*\cF)=\Sym^2\cF\ar[rr]\ar[rd]^{\bar{s}}&&p_*\cO_{\PP^2}(4)=\Sym^4W_3^*\otimes \cO_{\bK}\\
&p_*\cM=\cU.\ar[ur]&}
\end{equation}
We need to introduce a special locus in $\bK$ for the analysis of the induced morphism $\bar{s}$ in \eqref{pushdown}, which will turn out to be a degeneracy locus.
\begin{definition}\label{degdef}
Let us define $\bD_5$ as the locus of \emph{non-curvilinear} points in $\PP^2$ (with the reduced scheme structure) that are $\SL_3$-orbits of the ideal $I_{q}^2=\langle x^2,xy,y^2\rangle$ for the coordinate point $q=[0:0:1]$.
\end{definition}
In an extrinsic sense, one can expect that $\bD_5$ is isomorphic to $\Gr(2, W_3^*)$.
\begin{lemma}\label{deglem}
Let $\bD:=\Gr(2, W_3^*)$ be the Grassmannian of points in the projective plane $\PP^2=\PP(W_3)$. Then there exists an embedding $j:\bD\hookrightarrow \bK$ whose image is $\bD_5$.
\end{lemma}
\begin{proof}
Taking the second symmetric power of the universal sequence $\ses{\cD}{W_3^*\otimes \cO_{\bD}}{\cQ}$ of the Grassmannian $\bD$, we obtain the following sequence:
\begin{equation}\label{fatseq}
0\longrightarrow \Sym^2\cD\longrightarrow \Sym^2W_3^* \otimes \cO_{\bD} \longrightarrow (W_3^* \otimes \cO_{\bD})\otimes \cQ\longrightarrow 0,
\end{equation}
which defines an embedding map into $\Gr(3,\Sym^2W_3^*)$. As checking the image fiberwisely, the embedding map factors through $\bD_5\subset \bK$.
\end{proof}
\begin{remark}\label{wuniv}
When we identify $\bK$ with the Hilbert scheme $\bH[3]$ of three points in $\PP^2=\PP^2(W_3)$, one can obtain a family of subschemes of \emph{non-curvilinear} points parametrized by $\bD$.
Let $\cW$ be the universal family of points in $\PP^2$ over $\bD$. Then the relative Koszul complex of the ideal sheaf $\cI_{\cW}$ is acyclic:
\begin{equation}\label{symuniv}
0\lr \wedge^2\cD \boxtimes \cO_{\PP^2}(-2)\lr\cD \boxtimes \cO_{\PP^2}(-1)\lr \cI_{\cW}\lr 0.
\end{equation}
By taking the second symmetric power of the sequence \eqref{symuniv}, we have an exact sequence
\[
0\lr (\wedge^2\cD\otimes \cD)\boxtimes \cO_{\PP^2}(-3)\lr \Sym^2\cD \boxtimes \cO_{\PP^2}(-2)\lr \Sym^2\cI_{\cW}=\cI^2_{\cW}\lr 0.
\]
Since $\cI^2_{\cW}\subset \cI_{\cW}\subset \cO_{\bD\times \PP^2}$, one can easily check that $\cI^2_{\cW}$ is the flat family of non-curvilinear points.
\end{remark}

Now we are ready to state and prove one of the main propositions. Recall that $\bR_4$ is the closure of the space of rational quartic curves in $\PP^2$ under $\bM_4$. Also $\bR_4\subset \bQ_4\setminus \bF^{-}=\PP(\cU)\setminus \bF^{-}$ where $\bF^{-}$ is the flipping locus in $\bQ_4$ (Section \ref{fliplocus}).
\begin{proposition}\label{mainprop1}
Let $\cV:=\Sym^2\cF$ in equation \eqref{pushdown} and $\cV^{\circ}$ be the restriction of $\cV$ along the open subset $\bK\setminus \bD_5$. Then there exists an embedding map
\[
\Psi:\PP(\cV^{\circ}) \xhookrightarrow{}{} \PP(\cU)\setminus \bF^{-}
\]
such that the image of $\Psi$ is birational to the space $\bR_4$.
\end{proposition}
\begin{proof}
Note that $\bD_3\cap \bD_5=\emptyset$ by its definition. For each point $[Z]\in \bK\setminus (\bD_3\sqcup \bD_5)$, the map $\bar{s}$ in \eqref{pushdown} is nothing but the canonical isomorphism map \[\cV_{[Z]}=\Sym^2\rH^0(I_Z(2)) \cong \rH^0(I_Z^2(4))\subset \rH^0(I_Z(4))=\cU_{[Z]},\] where the second map comes from the inclusion $I_Z^2\subset I_Z$.
Since we can easily check that $h^0(I_Z^2(4))=6$ for each representative $[Z]$ of the $\SL_3$-orbits, $\bar{s}$ is injective.

For the closed point $[l]\in \bD_3(=\Gr(1,W_3^*))$, then the fiber of $\cU_{[l]}=\rH^0(E_{[l]}(2))$ (cf. \eqref{eq011}) fits into the exact sequence
\[
\ses{\rH^0(\cO_{l}(1))}{\cU_{[l]}}{\rH^0(\cO_{\PP^2}(3))}.
\]
Also the map $\cU_{[l]}\lr p_*\cO_{\PP^2}(4)_{[l]}$ in \eqref{pushdown} factors through the multiplication map $l\cdot(-):\rH^0(\cO_{\PP^2}(3))\lr \rH^0(\cO_{\PP^2}(4))\cong\Sym^4W_3^*$ (cf. \eqref{fact1}). Up to an $\SL_3$-equivalence, let $l=\langle x\rangle$. Also we assume that $\rH^0(\cF_{[l]})=\text{span}_{\CC}(x^2,xy,xz)$. Then the image of $\cV_{[l]}$ by the map $\bar{s}_{[l]}$ in $\Sym^4W_3^*$ is $x^2\cdot \Sym^2W_3^*$ and thus it is injective. Additionally, the image has a nonzero cubic part and hence it lies the complement $\PP(\cU)\setminus \bF^{-}$ ((2) of Proposition \ref{sum2}).

Lastly, for the ideal $I_{q}=\langle x, y\rangle\in \bD(\cong\bD_5)$, the image by $\bar{s}$ is given by $\rH^0(\PP^2, I_{q}^4(4))\cong \Sym^4\langle x, y\rangle$ and thus it is not injective. Since $\bD_5$ is transitive under the $\SL_3$-action, this property holds for every point in $\bD_5$.
\end{proof}
From the proof of Proposition \ref{mainprop1}, the map $\bar{s}$ in \eqref{pushdown} is generically injective (hence injective as sheaf homomorphism). Furthermore, the map $\bar{s}$ decomposes into the following way. By taking the symmetric power of the equation \eqref{univinc} and projecting into the first component (Section \ref{stdsl3}), we have
\[
\xymatrix{\cV=\Sym^2\cF\ar@{^{(}->}[r]^>>>>>>>i\ar@/_2.5pc/[ddr]_{p_1\circ i}&\Sym^2(\Sym^2W_3^*)\otimes \cO_{\bK}\ar@{=}[d]\\
&(\Sym^4 W_3^*\oplus \Sym^2(\wedge^2 W_3^*))\otimes \cO_{\bK}\ar[d]^{p_1}\\
&\Sym^4 W_3^*\otimes \cO_{\bK}}.
\]
By its construction, the map $\bar{s}$ in \eqref{pushdown} is simply the composition of the map $p_1$ followed by the inclusion map $i$. The proof of Proposition \ref{mainprop1} tells us that the map $\bar{s}=p_1\circ i$ is injective outside $\bD_5$.
\subsection{Modification of the symmetrized bundle via ideal flattening}\label{modsec}
In this subsection, by modifying the bundle $\cV=\Sym^2\cF$ on $\bK$ along $\bD_5$, we obtain a subbundle $\cV'$ of $\cU$ such that the quotient sheaf $\cU/\cV'$ is locally free (Theorem \ref{mainthm}). From this, we can deduce the existence of a morphism $\PP(\cV')\lr \PP(\cU)$ and therefore a morphism to the space $\bR_4$.

Let $\im((\psi\otimes \text{id})\circ k)$ be the image of $\Sym^2\cF$ in \eqref{univcom}. Since our modification will be done along the complement of $\bD_3=\Gr(1,W_3^*)$ in $\bK$ (Section \ref{hilbpt}), we can assume that $\im((\psi\otimes \text{id})\circ k)\subset \cM\cong \cI_{\cZ}(4)$ where $\cZ$ is the universal subscheme over $\bK$ and thus the image becomes a twisted ideal sheaf
\[
\im((\psi\otimes \text{id})\circ k)\cong \cI_{\cS}(4)\subset \cI_{\cZ}(4)\subset \cO_{\bK}\boxtimes \cO_{\PP^2}(4)
\]
for a subscheme $\cS\subset\bK\times \PP^2$. As shown in the proof of Proposition \ref{mainprop1}, the fibers over $\bK$ are determined by
\begin{equation}\label{do109}
\cI_{\cS}|_{[Z]\times\PP^2}=
\begin{cases}
I_Z^2,\;[Z]\in\bK\setminus \bD_5,\\
I_{q}^4,\; [Z]=[q^2]\in \bD_5.
\end{cases}
\end{equation}
From now on, we will denote $q^k$ by the subscheme defined by the ideal $I_q^k$ for a point $q\in\PP^2$.
Since the first (resp. second) case in \eqref{do109} has constant Hilbert polynomial $10$ (resp. $9$), $\cS$ is a flat family of subschemes in $\PP^2$ over $\bK\setminus \bD_5$. We will blow up $\bK$ along $\bD_5$ and extend the flat family as follows: Let $\pi: \widetilde{\bK}\lr \bK$ be the blow-up of $\bK$ along $\bD_5$. Pull back $\cS$ along the map $\pi\times \text{id}:\widetilde{\bK}\times \PP^2\lr \bK\times \PP^2$, and denote the resulting scheme by $\widetilde{\cS}$. Let $\widetilde{\bD}_{5}$ be the exceptional divisor of $\bD_5$, and let $\widetilde{\cS}_{0}$ be the restriction of $\widetilde{\cS}$ over $\widetilde{\bK}\setminus \widetilde{\bD}_{5}$. Finally, let us denote $\overline{\cS}$ as the \emph{scheme-theoretic} image of $\widetilde{\cS}_{0}$ in $\widetilde{\bK}\times \PP^2$.
\[
\xymatrix{\overline{\cS}:=\overline{\widetilde{\cS}_{0}}\ar@{^{(}->}[r]\ar[rd]&\widetilde{\bK}\times \PP^2\ar[d]&\widetilde{\bD}\times \PP^2\ar[r]^{\pi|_{\widetilde{\bD}}\times \text{id}}\ar[d]\ar@{_{(}->}[l]_{\widetilde{j}\times{\text{id}}}&\bD\times\PP^2\ar@/^1.5pc/[ddl]\\
&\widetilde{\bK}\ar[d]^{\pi}&\widetilde{\bD}:=\bD\times_{\bK}\widetilde{\bK}\ar@{_{(}->}[l]_<<<<<{\widetilde{j}}\ar[d]^{\pi|_{\widetilde{\bD}}}&\\
&\bK&\bD.\ar@{_{(}->}[l]_{j}&}
\]
Let $\widetilde{\cZ}=(\pi\times \text{id})^{-1} (\cZ)$ be the pull-back of the universal subscheme $\cZ$ along the map $\pi\times \text{id}$. Recall that $j:\bD\hookrightarrow \bK$ is the embedding map (Lemma \ref{deglem}). Also $\cW$ is the universal family of single point in $\PP^2$ parameterized by $\bD$ (Remark \ref{wuniv}).
\begin{proposition}\label{mainprop2}
Under the above notations and definitions, there exists an \emph{nested} inclusion
\begin{equation}\label{netin}
\widetilde{\cZ}\subset \overline{\cS}\subset\widetilde{\cS}\subset\widetilde{\bK}\times \PP^2
\end{equation}
such that $\overline{\cS}$ is a flat family of subschemes in $\PP^2$ over $\widetilde{\bK}$.
\end{proposition}
\begin{proof}
The flat extension of subschemes can be regarded as a modification of sheaves (or complexes).
Let $\widetilde{\cW}:=(\pi|_{\widetilde{\bD}}\times{\text{id}})^*(\cW)$.
In our setting, let
\[
\cG=\text{Ker}\left\{\cO_{\widetilde{\cS}}\oplus (\widetilde{j}\times{\text{id}})_*\cO_{\widetilde{\cW}}[1]\twoheadrightarrow \cO_{\widetilde{\cS}}|_{\widetilde{\bD}_5\times\PP^2}\oplus \cO_{\widetilde{\cW}}[1]\twoheadrightarrow \cO_{\widetilde{\cS}}|_{\widetilde{\bD}_5\times\PP^2}\right\}
\]
be the elementary modification of the complex $\cO_{\widetilde{\cS}}\oplus (\widetilde{j}\times{\text{id}})_*\cO_{\widetilde{\cW}}[1]$ along $\widetilde{\bD_5}\times\PP^2\cong \widetilde{\bD}\times\PP^2$ by $\cO_{\widetilde{\cS}}|_{\widetilde{\bD}_5\times\PP^2}$ (\cite[Proposition 2.1]{Lo13}). The effect of the modification is to change the sub/quotient sheaf. Specifically, for each closed point $(q^2,v)\in \widetilde{\bD}_5=\PP(\cN_{\bD_5|\bK})$, $[I_q]=[W_2^*=\langle x,y\rangle] \in \Gr(2,W_3^*)$, the modified complex fits into the following exact sequence:
\[
0\lr \cO_{q^4}\lr \cG|_{(p^2,v)\times \PP^2}\lr \CC_q[1]\lr0.
\]
We show that the complex $\cG|_{(q^2,v)\times \PP^2}$ is isomorphic to a non-split one (i.e., structure sheaf) by analyzing the Kodaira-Spencer map as below. 
\begin{equation}\label{ksdia}
\xymatrix{
T_{q^2} \ar[dd]\bK\ar[r]^<<<<{\text{KS}}& \Ext_{\PP^2}^1(\cO_{q^4}\oplus \CC_q[1], \cO_{q^4}\oplus \CC_q[1]) \ar[d]\\
&\Ext_{\PP^2}^1(\CC_q[1], \cO_{q^4}\oplus \CC_q[1])\ar[d] \\
N_{\bD_5|\bK,q^2}\ar[r]^<<<<<<<<{\overline{\text{KS}}}&\Ext_{\PP^2}^1(\CC_q[1],  \cO_{q^4}).}
\end{equation}
One can easily see that the map $\text{KS}$ in \eqref{ksdia} descents to the normal space $N_{\bD_5|\bK, q^2}$. Also it induces a linear map $\overline{\text{KS}}$ into
\[
\Ext_{\PP^2}^1(\CC_q[1],  \cO_{q^4})\cong \Hom_{\PP^2}(\CC_q,  \cO_{q^4})\cong\Ext^1_{\PP^2}(\CC_q,  I_{q^4})\cong (\wedge^2 W_2^*)^{\otimes 2}\otimes \Sym^3 W_2^*,
\] 
where the last isomorphism comes from the resolution
\[
0\longrightarrow(\wedge^2 W_2^*\otimes \Sym^3 W_2^*) \boxtimes \cO_{\PP^2}(-5)\longrightarrow \Sym^4 W_2^* \boxtimes \cO_{\PP^2}(-4)\longrightarrow I_{q}^4\longrightarrow 0
\]
of $I_{q}^4$ obtained by taking the fourth symmetric power of the Koszul complex of the ideal sheaf $I_{q}$ (cf. Remark \ref{wuniv}).
However, according to Lemma \ref{norbundle} below, the normal space $N_{\bD_5|\bK, q^2}$ can be identified with $\Sym^3W_2$.
Since the diagram \eqref{ksdia} is $\SL(W_2)$-equivariant and both spaces in the bottom row are irreducible representations of $\SL_2$, it suffices to check that the induced map $\overline{\text{KS}}$ is non-zero at a special point. For instance, consider the curve-linear curve $\cI_{\cZ_t}=\langle x^2,xy,y^2+tzx\rangle$ approaching the point $I_{q}^2$. Clearly for all $t\in \CC^*\subset \bK\setminus\bD_5$ and $v=2\partial_x^2\partial_y\in \Sym^3W_2$. The image of $v$ by $\bar{\text{KS}}$ is $\bar{\text{KS}}(v)=\widetilde{\CC[x,y,z]/J}$ for $J=\langle x^4, x^3y, x^2y^2, xy^3, y^4\rangle+\langle zx^2y\rangle$. Since $I_p^4\subset J\subset I_p^2$, we have the inclusions of subschemes in \eqref{netin} by the existence of relative Quot scheme.
\end{proof}

\begin{lemma}\label{norbundle}
The normal bundle of $\bD_5$ in $\bK$ is given by
\[
\cN_{\bD_5\mid\bK}\cong \Sym^3 \cD^*,
\]
where $\cD$ is the universal subbundle of $\bD_5\cong \bD=\Gr(2,W_3^*)$.
\end{lemma}
\begin{proof}
Let $[W_2^*]\in \Gr(2,W_3^*)=\bD\cong\bD_5$ denote the point representing the subscheme defined by the ideal $\langle x,y\rangle^2$. Since the space $\bK$ around $\bD_5$ is isomorphic to $\bH[3]$, we can apply the deformation theory of the Hilbert scheme of points. It is well-known that 
\[\begin{split}
T_{[\langle x,y\rangle^2]} \bK&\cong T_{[\langle x,y\rangle^2]}\bH[3]\\
&\cong \Hom_{\CC[x,y,z]}( \langle x^2, xy, y^2\rangle, \CC[x,y,z]/\langle x^2, xy, y^2\rangle).
\end{split}
\]
Since $\bD_5$ is $\SL_3$-transitively invariant, by varying the space $W_2^*$, we get a canonical isomorphism
\[
\cT_{\bK}|_{\bD_5}\cong \cH om_{\bD_5} (\Sym^2\cD,\cD\otimes\cQ)\cong(\Sym^2\cD^*)\otimes (\cD\otimes \cQ).
\]
On the other hand, from the universal sequence $\ses{\cD}{W_3^*\otimes \cO_{\bD_5}}{\cQ}$, we obtain $T_{\bD_5}\cong\cD^*\otimes \cQ$. As combining these ones, we can construct a commutative diagram
\[
\xymatrix{\cT_{\bD_5}\ar@{^{(}->}[r]\ar@{=}[d]&\cT_{\bK}|_{\bD_5}\ar@{->>}[r]\ar@{=}[d]&\cN_{\bD_5\mid\bK}\ar@{=}[d]\\
\cD^*\otimes \cQ\ar@{^{(}->}[r]&(\Sym^2\cD^*\otimes \cD)\otimes\cQ\ar@{->>}[r]&\Sym^3\cD^*.}
\]
Here the below horizontal sequence is the exact sequence $\ses{\cD^*}{\Sym^2\cD^*\otimes \cD}{\Sym^3\cD^*\otimes \cQ^*}$ tensored with $\cQ$, where the first map is given by
\[\begin{split}
\partial_{x}&\mapsto \partial_{x}^2\otimes y- \partial_{x}\cdot \partial_{y} \otimes x,\\
\partial_{y}&\mapsto \partial_{y}^2\otimes x- \partial_{y}\cdot \partial_{x} \otimes y 
\end{split}
\]
such that the dual of $x$ (resp. $y$) is denoted by $\partial_{x}$ (resp. $\partial_{y}$).
\end{proof}
From the inclusion of the statement in Proposition \ref{mainprop2}, we have a commutative diagram:
\[
\xymatrix{\Sym^2\cF\ar@{->>}[rd]\ar[rr]&&\cI_{\overline{\cS}}(4)\subset \cI_{\widetilde{\cZ}}(4)\\
&\cI_{\widetilde{\cS}}(4)\ar@{^{(}->}[ru]&}
\]
over $\widetilde{\bK}\times \PP^2$. By pushing down this diagram along the projection map $p:\widetilde{\bK}\times \PP^2\lr \widetilde{\bK}$, we have the followings.
\begin{theorem}\label{mainthm}
The direct image sheaf $\cV' := p_*\cI_{\overline{\cS}}(4)$ is locally free. Furthermore, $\cV'$ fits into a short exact sequence
\[\ses{\cV'}{p_*\cI_{\widetilde{\cZ}}(4) = \pi^*\cU}{\pi^*\cU/\cV' := \cQ}\]
where the quotient sheaf $\cQ$ is also locally free. This implies the existence of an injective morphism \[\PP(\cV') \subset \PP(\pi^*\cU)\]over $\widetilde{\bK}$ and thus a morphism into $\bR_4$.\end{theorem}
\begin{proof}
The locally freeness of $\cV'$ is due to the fact that the family $\overline{\cS}$ is flat over $\widetilde{\bK}$ and $R^ip_*\cI_{\overline{\cS}}(4) = 0$ for all $i\geq 1$. Moreover, from the nested ideal sequence $\ses{\cI_{\overline{\cS}}(4)}{\cI_{\widetilde{\cZ}}(4)}{\cI_{\widetilde{\cZ}\mid\bar{\cS}}(4)}$ (\eqref{netin} in Proposition \ref{mainprop2}) and $R^ip_*\cI_{\widetilde{\cZ}}(4) = 0$ for all $i\geq 1$, the quotient sheaf $\cQ$ is locally free.
\end{proof}
To compute the Chow ring of $\PP(\cV')$, we refer to the following proposition.
\begin{proposition}\label{chowset}
There exists an exact sequence
\[
\ses{\cV}{\cV'}{\cV'/\cV=:\cA},
\]
where $\cA$ is isomorphic to 
\[\cA\cong \cO_{\widetilde{\bD}_5}(\widetilde{\bD}_5)\otimes\pi|_{\widetilde{\bD}_5}^*\cO_{\bD_{5}}(-2)\]
for the restriction map $\pi|_{\widetilde{\bD}_5}: \widetilde{\bD}_5\lr \bD_5$ of the blow-up map $\pi$.
\end{proposition}
\begin{proof}
Note that $\text{Supp}(\cA) \subset \widetilde{j}(\widetilde{\bD})=\widetilde{\bD}_5$ for the inclusion map $\widetilde{j}: \widetilde{\bD}\hookrightarrow \widetilde{\bK}$. By pushing down the exact sequence $\Sym^2\cF \to \cI_{\bar{S}}(4) \to \cI_{\bar{S}|\widetilde{S}}(4) \to 0$ along the projection map $p: \widetilde{\bK} \times \PP^2 \to \widetilde{\bK}$, we obtain an exact sequence
\[0 \to \cV \to \cV' \to p_*\cI_{\bar{S}|\widetilde{S}}(4),\]
where the injection comes from the construction of locally free sheaves. On the other hand, by Proposition \ref{mainprop2}, we obtain a commutative diagram
\[
\xymatrix{
\cO_{\widetilde{\bK}}(\widetilde{\bD}_5)|_{\widetilde{\bD}_5}\otimes \cA\cong \mathcal{T}or_{\widetilde{\bK}}^1(\cA, \cO_{\widetilde{\bD}_5})\ar@{^{(}->}[r]\ar@{=}[d] & \cV|_{\widetilde{\bD}_5}\ar[r]\ar@{=}[d] & \cV'|_{\widetilde{\bD}_5}\ar@{->>}[r]& \cA|_{\widetilde{\bD}_5} \\
\cO_{\bD}(-2)\cong \Sym^2(\wedge^2 \cD)\ar@{^{(}->}[r] & \Sym^2(\Sym^2(\cD))\ar@{->>}[r] & \Sym^4\cD,\ar@{^{(}->}[u] &
}
\]
where the bottom row corresponds to the push-forward along $\widetilde{j}_*$ of the standard exact sequence of symmetric powers on $\widetilde{\bD}$. Thus we complete the proof.
\end{proof}
\begin{corollary}\label{marcor}
The bundle $\pi^*\cV$ is the elementary transformation of $\cV'$ by $\cA$ along the divisor $\widetilde{\bD}_5$. Hence there exists a flip between $\PP(\pi^*\cV)$ and $\PP(\cV')$ over $\widetilde{\bK}$ by the general result of Maruyama \cite{Mar73}.
\end{corollary}
From Theorem \ref{mainthm} and Corollary \ref{marcor}, we have a commutative diagram
\begin{center}
\begin{tikzcd}
\PP(\pi^*\cV) \arrow[rd] \arrow[dd, Rightarrow] \arrow[rr, dashed, "\text{flip}"] &                                                          & \PP(\cV') \arrow[rd] \arrow[dd, Rightarrow] \arrow[ll, dashed]&                                 \\
                                                                   & \PP(\cV) \arrow[rr, dashed, hook] \arrow[dd, Rightarrow] &                                                   & \PP(\cU) \arrow[dd, Rightarrow] \\
\widetilde{\bK} \arrow[rd, "\pi"] \arrow[r, Rightarrow, no head]   & {} \arrow[r, Rightarrow, no head]                         & \widetilde{\bK} \arrow[rd, "\pi"]                 &                                 \\
                                                                   & \bK \arrow[rr, Rightarrow, no head]                       &                                                   & \bK.                            
\end{tikzcd}
\end{center}
\subsection{Description of the contraction map $\PP(\cV')\lr \bR_4$}

In this subsection, we analyze the contraction of the induced map $\PP(\cV')\lr \bR_4\subset \PP(\cU)$.
\begin{proposition}
The birational morphism $\PP(\cV')\lr \bR_4$ contracts the exceptional divisor (i.e., a $\PP^5$-bundle over $\widetilde{\bD}_5$) to sublocus of $\bR_4$, which is a $\PP^8$-fiberation over $\PP^2\cong \bD_5$ ((1) of Proposition \ref{sum2}). 
\end{proposition}
\begin{proof}
Let us fix a point $[I_q^2]=[\langle x,y\rangle^2] \in \bD_5$. The exceptional divisor $\widetilde{\bD}_5$ over $[I_q^2]$ has the following description. Let $t_0, t_1, t_2, t_3$ be the local coordinate of the exceptional fiber $\PP(\Sym^3W_2)$ over the point $[I_q^2]\in \bD_5$. Also let $u_0, u_1, u_2, u_3, u_4$ and $k$ be the coordinates of the fiber $\PP^5(\cV'|_{[I_q^2,t_i]})$. Then the quartic forms parameterized by the fiber $\pi|_{\widetilde{\bD}_5}^{-1}([I_q^2])$ are given by the first row in
\begin{equation}\label{canform}
\begin{split}
u_0x^4+u_1x^3y+u_2x^2y^2+u_3 xy^3+u_4y^4&+k(t_0zx^3+t_1zx^2y+t_2zxy^2+t_3zy^3)\\
=u_0x^4+u_1x^3y+u_2x^2y^2+u_3 xy^3+u_4y^4&+u_5zx^3+u_6zx^2y+u_7zxy^2+u_8zy^3.
\end{split}
\end{equation}
On the other hand, let $u_0,u_1,\cdots$, $u_7$ and $ u_8$ be the homogenous coordinates of the projective space $\PP^8$, which parameterizes the quartic forms in the second row of \eqref{canform}. Let $\PP^4\subset \PP^8$ be a linear space defined by $u_5=u_6=u_7=u_8=0$. Then the blow-up of $\PP^8$ along the linear subspace $\PP^4$ is isomorphic to the closure of the graph of the rational map $\PP^8\dashrightarrow \PP^3$, $[u_0:u_1:...:u_8]\mapsto [u_5:u_6:u_7:u_8]$. That is, the closure is the subvariety:
\[
\cC:=\left\{[u_0:u_1:...:u_8]\times [t_0:t_1:t_2:t_3]\;\mid\; \rk\begin{bmatrix}u_5&u_6&u_7&u_8\\ t_0&t_1&t_2&t_3\end{bmatrix}=1 \right\}\subset \PP^8\times\PP^3.
\]
The projection $\cC\lr \PP^3$ into the second component has a projective bundle structure with the fiber $\PP^5$, which is exactly the fiber $\pi|_{\widetilde{\bD}_5}^{-1}([I_q^2])$. This justifies the equality of \eqref{canform} and describes the contraction map $\widetilde{\bD}_5\lr \bR_4$ such that the image is a $\PP^8$-fiberation over $[I_q^2]\in\bD_5$. Note that the fiber of $\cU$ over $[I_q^2]$ is $\PP(\cU_{[I_q^2]})=\PP(\rH^0(I_q^2(4)))$.
\end{proof}

\begin{thebibliography}{CKPU13}

\bibitem[BF97]{BF97}
Kai Behrend and Barbara Fantechi.
\newblock The intrinsic normal cone.
\newblock {\em Invent. Math.}, 128(1):45--88, 1997.

\bibitem[BMW14]{BMW14}
Aaron Bertram, Cristian Martinez, and Jie Wang.
\newblock The birational geometry of moduli spaces of sheaves on the projective
  plane.
\newblock {\em Geom. Dedicata}, 173:37--64, 2014.

\bibitem[CC11]{CC11}
Dawei Chen and Izzet Coskun.
\newblock Towards {M}ori's program for the moduli space of stable maps.
\newblock {\em Amer. J. Math.}, 133(5):1389--1419, 2011.

\bibitem[CCM16]{CCM16}
Jinwon Choi, Kiryong Chung, and Mario Maican.
\newblock {Moduli of sheaves supported on Quartic space curves}.
\newblock {\em Michigan Mathematical Journal}, 65(3):637 -- 671, 2016.

\bibitem[CHL18]{CHL18}
Kiryong Chung, Jaehyun Hong, and SangHyeon Lee.
\newblock Geometry of moduli spaces of rational curves in linear sections of
  grassmannian gr(2,5).
\newblock {\em Journal of Pure and Applied Algebra}, 222(4):868 -- 888, 2018.

\bibitem[Chu11]{Chu11}
Kiryong Chung.
\newblock {\em Compactified moduli spaces of the rational curves in a
  projective variety}.
\newblock PhD thesis, Seoul National University, 2011.

\bibitem[Chu21]{Chu21}
Kiryong Chung.
\newblock {Sheaf Theoretic Compactifications of the Space of Rational Quartic
  Plane Curves}.
\newblock {\em Taiwanese Journal of Mathematics}, 25(3):463 -- 476, 2021.

\bibitem[Chu22]{Chu22}
Kiryong Chung.
\newblock Desingularization of Kontsevich's compactification of twisted cubics in $V_5$.
\newblock {\em manuscripta mathematica}, 2022.

\bibitem[CK11]{CK11}
Kiryong Chung and Young-Hoon Kiem.
\newblock Hilbert scheme of rational cubic curves via stable maps.
\newblock {\em Amer. J. Math.}, 133(3):797--834, 2011.

\bibitem[CKPU13]{CKPU13}
David Cox, Andrew~R. Kustin, Claudia Polini, and Bernd Ulrich.
\newblock A study of singularities on rational curves via syzygies.
\newblock {\em Mem. Amer. Math. Soc.}, 222(1045), 2013.

\bibitem[CM17]{CM17}
Kiryong Chung and Han-Bom Moon.
\newblock {Chow Ring of the Moduli Space of Stable Sheaves Supported on Quartic
  Curves}.
\newblock {\em The Quarterly Journal of Mathematics}, 68(3):851--887, 02 2017.

\bibitem[DH88]{DH88b}
Steven Diaz and Joe Harris.
\newblock Geometry of severi varieties, ii: Independence of divisor classes and
  examples.
\newblock In {\em Algebraic Geometry Sundance 1986}, pages 23--50, Berlin,
  Heidelberg, 1988. Springer Berlin Heidelberg.

\bibitem[LFV21]{LFV21}
Andrea D.~Lorenzo, Damiano Fulghesu, and Angelo Vistoli.
\newblock The integral Chow ring of the stack of smooth non-hyperelliptic
  curves of genus three.
\newblock {\em Transactions of the American Mathematical Society},
  374(8):5583--5622, 2021.

\bibitem[EN62]{EN62}
John~A. Eagon and Douglas~G. Northcott.
\newblock Ideals defined by matrices and a certain complex associated with
  them.
\newblock {\em Proceedings of the Royal Society of London. Series A,
  Mathematical and Physical Sciences}, 269(1337):188--204, 1962.

\bibitem[FH91]{FH91}
William Fulton and Joe Harris.
\newblock {\em Representation theory}, volume 129 of {\em Graduate Texts in
  Mathematics}.
\newblock Springer-Verlag, New York, 1991.
\newblock A first course, Readings in Mathematics.

\bibitem[FP97]{FP97}
William Fulton and Rahul Pandharipande.
\newblock Notes on stable maps and quantum cohomology.
\newblock In {\em Algebraic geometry---{S}anta {C}ruz 1995}, volume~62 of {\em
  Proc. Sympos. Pure Math.}, pages 45--96. Amer. Math. Soc., Providence, RI,
  1997.

\bibitem[FT04]{FT04}
Hans~Georg Freiermuth and G{{\"u}}nther Trautmann.
\newblock On the moduli scheme of stable sheaves supported on cubic space
  curves.
\newblock {\em Amer. J. Math.}, 126(2):363--393, 2004.

\bibitem[Har86]{Har86}
Joe Harris.
\newblock On the severi problem.
\newblock {\em Inventiones mathematicae}, 84(3):445--461, 1986.

\bibitem[He98]{He98}
Min He.
\newblock Espaces de modules de syst{\`e}mes coh{\'e}rents.
\newblock {\em Internat. J. Math.}, 9(5):545--598, 1998.

\bibitem[Hir64]{Hir64}
Heisuke Hironaka.
\newblock Resolution of singularities of an algebraic variety over a field of
  characteristic zero: I.
\newblock {\em Annals of Mathematics}, 79(1):109--203, 1964.

\bibitem[Hui79]{Hui79}
Chung-Man Hui.
\newblock Plane quartic curves.
\newblock {\em PhD thesis, Liverpool}, 1979.

\bibitem[Kin94]{Kin94}
Alastair~D. King.
\newblock Moduli of representations of finite-dimensional algebras.
\newblock {\em Quart. J. Math. Oxford Ser. (2)}, 45(180):515--530, 1994.

\bibitem[Kon95]{Kon95}
Maxim Kontsevich.
\newblock Enumeration of rational curves via torus actions.
\newblock In Robbert~H. Dijkgraaf, Carel~F. Faber, and Gerard B.~M. van~der
  Geer, editors, {\em The Moduli Space of Curves}, pages 335--368, Boston, MA,
  1995. Birkh{\"a}user Boston.

\bibitem[Lo13]{Lo13}
Jason Lo.
\newblock Moduli of {PT}-semistable objects {II}.
\newblock {\em Trans. Amer. Math. Soc.}, 365(9):4539--4573, 2013.

\bibitem[LQZ03]{LQZ03}
Wei-Ping Li, Zhenbo Qin, and Qi~Zhang.
\newblock Curves in the {H}ilbert schemes of points on surfaces.
\newblock In {\em Vector bundles and representation theory ({C}olumbia, {MO},
  2002)}, volume 322 of {\em Contemp. Math.}, pages 89--96. Amer. Math. Soc.,
  Providence, RI, 2003.

\bibitem[Mar73]{Mar73}
Masaki Maruyama.
\newblock On a family of algebraic vector bundles.
\newblock {\em Number Theory, Algebraic Geometry, and Commutative Algebra},
  pages 95--149, 1973.

\bibitem[Sim94]{Sim94}
Carlos~T. Simpson.
\newblock Moduli of representations of the fundamental group of a smooth
  projective variety. {I}.
\newblock {\em Inst. Hautes {\'E}tudes Sci. Publ. Math.}, (79):47--129, 1994.

\bibitem[Ste54]{Ste54}
Jakob Steiner.
\newblock Allgemeine eigenschaften der algebraischen curven.
\newblock {\em Journal f{\"u}r die reine und angewandte Mathematik (Crelles
  Journal)}, 1854(47):1--6, 1854.

\bibitem[Tha96]{Tha96}
Michael Thaddeus.
\newblock Geometric invariant theory and flips.
\newblock {\em J. Amer. Math. Soc.}, 9(3):691--723, 1996.

\bibitem[Woo13]{Woo13}
Matthew Woolf.
\newblock Nef and effective cones on the moduli space of torsion sheaves on the
  projective plane.
\newblock arXiv:1305.1465, 2013.

\bibitem[Zeu73]{Zeu73}
Hieronymus~G. Zeuthen.
\newblock {\em Almindelige Egenskaber ved Systemer af plane Kurver, med
  Anvendelse til Bestemmelse af Karakteristikerne i de element{\ae}re Systemer
  af fjerde Orden}.
\newblock Det Kongelige Danske videnskabernes selskabs skrifter. 5te R{\ae}kke
  : Naturvidenskabelig og mathematisk afdeling. Luno, 1873.

\end{thebibliography}

\end{document}